\documentclass[a4paper,12pt,draft]{article}
\usepackage[T1]{fontenc}
\usepackage{amssymb,amsmath,amsthm,latexsym}
\usepackage{soul}
\usepackage[normalem]{ulem}
\usepackage{color}

\usepackage[a4paper,margin=1in]{geometry} 
\usepackage{verbatim}
\usepackage{dsfont}
\usepackage{marginnote}

\theoremstyle{plain}
\newtheorem{thm}{Theorem}[section]
\newtheorem{prop}[thm]{Proposition}
\newtheorem{corollary}[thm]{Corollary}
\newtheorem{lemma}[thm]{Lemma}

\theoremstyle{definition}

\newtheorem{remark}[thm]{Remark}

\numberwithin{equation}{section}

\newcommand{\rd}{\mathbf{R}^d}

\newcommand\R{\mathbf{R}}
\newcommand\E{\mathbf{E}}
\newcommand\Pp{\mathbf{P}}

\renewcommand{\le}{\leqslant}
\renewcommand{\leq}{\leqslant}
\renewcommand{\ge}{\geqslant}
\renewcommand{\geq}{\geqslant}

\newcommand{\1}{1}
 
\newcommand{\ds} {\, \mathrm{d}s}
\newcommand{\dt} {\, \mathrm{d}t}
\newcommand{\dx} {\, \mathrm{d}x}
\newcommand{\dy} {\, \mathrm{d}y}

\newcommand{\dr} {\, \mathrm{d}r}
\newcommand{\dist} {\mathrm{dist}}
\def\proofof{\noindent{\emph{Proof of}} }

\begin{document}\allowdisplaybreaks

\title{\Large\bfseries Large solutions to semilinear equations for subordinate Laplacians in $C^{1,1}$ bounded open sets}

\author{ Indranil Chowdhury\footnote{ Research supported by the Croatian Science Foundation under the project 4197 and in part under the project 2277 }
    \and
    Zoran Vondra\v cek$^\ast$
    \and
    Vanja Wagner$^\ast$
}

\date{}

\maketitle

 \begin{abstract}\noindent
  We study the existence of a large solution to a semilinear problem in a bounded open $C^{1,1}$ set for a class of nonlocal operators obtained by an appropriate subordination of the Laplacian. These operators are classical generalisations of the fractional Laplacian. The existence result is shown under a nonlocal version of the Keller-Osserman condition, stated in terms of the subordinator and the source term $f$. 

      31C05, 35J61, 45K05 
      
  \end{abstract}

\section{Introduction}\label{sec-INTRO}

Semilinear elliptic equations of the form
\[
\Delta u = f(u) \quad \text{in } \Omega,
\]
where \(\Omega\) is a domain in \(\mathbf{R}^d\), and \(f\) is a nonlinear function, have been the subject of intensive investigation for several decades, see, e.g., \cite{MV} for a comprehensive overview of results. One particularly interesting class of solutions to these equations is the so-called large solutions — those that blow up at the boundary of the domain. These solutions do not have a finite $M$-boundary trace (see \cite[Definition 1.3.6.]{MV}), and therefore cannot be dominated by nonnegative classical harmonic functions on $\Omega$. The existence, uniqueness, and qualitative behavior of such solutions have been thoroughly studied for local operators, and the Laplacian in particular, and necessary and sufficient conditions for their existence are well understood. A foundational result in this context is the Keller-Osserman condition introduced in \cite{keller} and \cite{osserman}, which provides a criterion for the existence of large solutions based on the growth of the nonlinearity \(f\).

In contrast, the analogous theory for equations involving nonlocal operators, even the simplest case of the fractional Laplacian $-(-\Delta)^{\alpha/2}$, $\alpha\in(0,2)$, has only recently begun to develop, e.g. \cite{CFQ}, \cite{FQ}, \cite{benchrouda}, \cite{abatangelo1}, \cite{abatangelo2}, \cite{BJK}, \cite{BVW}, \cite{klimsiak}, \cite{klimsiak25}. Semilinear equations involving fractional Laplacians and other nonlocal operators exhibit a rich and sometimes unexpected behavior that differs significantly from the local case. For example, nonlocal effects allow for boundary blow-up even in linear equations, a phenomenon not possible in classical Laplacian problems. Moreover,
the nonlocal nature of the operators leads to prescribing the analogue of the classical Dirichlet boundary conditions in terms of the exterior data on (parts of) $\Omega^c$, and boundary data given via an appropriate potential-theoretical trace operator, see \cite{BJK}, \cite{Bio} and \cite{BVW}.

The current work focuses on a class of semilinear problems involving a family of nonlocal operators defined via subordination of the Laplacian by a complete Bernstein function $\phi$, with the general form \( L = -\phi(-\Delta) \). These operators can be considered as generators of a special class of pure-jump Lévy processes, called subordinate Brownian motions, and have been studied extensively from both probabilistic and analytical perspectives. They are typically expressed as singular integral operators of the form
\[
Lu(x) = \text{P.V.} \int_{\mathbf{R}^d} (u(y) - u(x))j(|y - x|)\, dy,
\]
where the kernel \(j\) represents the jump intensity of the subordinated process. Under certain weak scaling assumptions on the subordinator $\phi$, the operator exhibits robust analytic properties and potential theory analogous to the classical fractional Laplacian case, i.e. for \( \phi(t) = t^{\alpha/2} \), $\alpha\in(0,2)$.
Our analysis relies on recent advances in potential theory for subordinate Brownian motions and the representation of the corresponding harmonic functions that allows for a precise formulation of the blow-up condition in terms of the renewal function corresponding to the conjugate subordinator associated with $\phi$. Among these, and recent advances in analytical properties of nonlocal operators and {(semi-)linear} equations, we highlight the works \cite{BKK, BJK, Bio, BVW, KKLL, GKL, BJ} and the references therein.

We consider solutions to the equation
\begin{equation}\label{e:semilinear}
-Lu = - f(u) \quad \text{in } \Omega,
\end{equation}
for a bounded open \( C^{1,1} \) set \( \Omega \subset \mathbf{R}^d \), \( d \geq 2 \), under a regime of strong boundary singularities, commonly referred to as large solutions. These are the solutions that exhibit blow-up at the boundary with a growth rate that surpasses that of any nonnegative $L$-harmonic function. In order to construct a large solution in the distributional sense, we rely on a generalization of the classical Keller-Osserman condition, adapted to the nonlocal context of this paper. In the classical setting of the Laplacian, the classical Keller-Osserman condition 
\begin{equation}\label{eq:KO_classical}
    \int_1^\infty \frac 1{\sqrt{F(t)}}\, dt<\infty,
\end{equation}
where $F$ is the primitive function of an increasing function $f$, is the necessary and sufficient condition for the existence of a solution to the semilinear equation
\begin{equation}\label{eq:local}
    \begin{cases}
   \ \ \  -\Delta u = -f(u), &\text{in }\Omega,\\
    \lim\limits_{\Omega\ni x\to\partial\Omega} u(x)=\infty.&   
    \end{cases}
\end{equation}

As mentioned above, the analogue of this equation in the nonlocal setting requires prescribing complement data, as well as an appropriate analogue of the boundary condition. The first result for the fractional Laplace operator in this direction is by Felmer and Quaas (\cite{FQ}), considered on a $C^2$ domain $\Omega$, of the form
\begin{equation}\label{eq:fractional}
    \begin{cases}
    \ \ \  (-\Delta)^{\alpha/2} u = -f(u), &\text{in }\Omega,\\
    \ \ \  u=g, & \text{in }\Omega^c\\
     \lim\limits_{\Omega\ni x\to\partial\Omega} u(x)=\infty.  & 
    \end{cases}
\end{equation}
for $f(t)=t^p$, $p>1$. Here they constructed boundary blow-up solutions in the viscosity sense by considering complement data $g$ with singularities at the boundary. Note that the problem \eqref{eq:fractional} is not the true analog of the boundary blow-up problem considered in the local case, cf. \eqref{eq:local}. Indeed, there exist $\alpha$-harmonic functions on $\Omega$ that explode at the boundary, with the rate at most $\delta_\Omega^{\frac\alpha2-1}$, where $\delta_\Omega(x)$ is the distance from the point $x\in \Omega$ to $\partial \Omega$. This means that, in the case of the fractional Laplace operator, solutions to \eqref{e:semilinear} that are dominated by corresponding harmonic functions, so-called \emph{moderate} solutions, can exhibit the same boundary blow-up. We are primarily interested in studying the true nonlocal analogue of \emph{large} solutions to \eqref{e:semilinear}, which in the fractional setting has a singularity at the boundary of higher order than $\delta_\Omega^{\frac\alpha2-1}$. Moreover, the corresponding boundary blow up should be the result of the interaction between the operator and the source term, and not with the complement data -- this is in similar spirit as the original results by Keller and Osserman in the local case. The resulting equation is then of the form 

\begin{equation}\label{eq:blow_up_fractional}
    \begin{cases}
       \ \ \ (-\Delta)^{\alpha/2} u = -f(u), &\text{in }\Omega,\\
   \ \ \   u=0,\qquad &\text{in }\Omega^c,\\
   \lim\limits_{\Omega\ni x\to \partial \Omega}\delta_\Omega^{1-\alpha/2}(x)u(x)=\infty. &
    \end{cases}
\end{equation}
In the follow-up work \cite{CFQ}, the authors show that problem \eqref{eq:blow_up_fractional} admits a large solution when $f(t)=t^p$ for the range $1 + \alpha\le  p\le 1 -\frac{\alpha}{\tau_0(\alpha)}$ where $\tau_0(\alpha) \in\langle- 1, 0]$. Moreover, they obtain the exact blow-up rate of order $\delta_\Omega(x)^{\frac{\alpha}{1-p}}$. Following their result, the authors in \cite{benchrouda} adopt a different approach based on analytic tools from potential theory for the symmetric $\alpha$-stable process, formulating the semilinear problem in a corresponding weak dual and distributional sense. In the case where $\Omega$ is a ball and $f(t)=t^p$, they construct an explicit nonnegative continuous solution to \eqref{eq:blow_up_fractional}, exhibiting the same blow-up rate as in \cite{CFQ}, for the range of parameters $1 + \alpha < p < \frac{2+\alpha}{2-\alpha}$. In addition, a nonexistence result is shown for $0 < p < 1 + \frac{\alpha}{2}$. The former range implies that the classical Keller-Osserman condition cannot be fully extended to the nonlocal setting, since \eqref{eq:KO_classical} is equivalent to $p>1$. Independently of \cite{benchrouda}, the question of the existence of solutions to \eqref{eq:blow_up_fractional} for more general power-like nonlinearities $f$ has been studied in \cite{abatangelo2}, under the term \emph{very large} solutions. Additionally, a somewhat incomplete fractional version of the Keller-Osserman condition is obtained, as a sufficient condition for the solvability of \eqref{eq:blow_up_fractional}. This condition is the key ingredient in obtaining a supersolution to \eqref{eq:blow_up_fractional}, which allows for a construction of a solution as an increasing limit of appropriately chosen moderate solutions to the semilinear problem \eqref{e:semilinear}. A slight drawback of this approximation approach is the lack of an exact blow-up rate of the obtained large solution.

In this paper, we focus on a class of semilinear problems involving more general subordinate Laplacians in bounded open \(C^{1,1}\) sets \(\Omega \subset \mathbf{R}^d\), \(d \geq 2\). Specifically, we study the problem
\begin{equation}\label{eq:problem}
\begin{cases}
\ -Lu = -f(u) & \text{in } \Omega, \\
\ \ \ u = 0 & \text{in } \mathbf{R}^d \setminus \Omega, \\
\lim\limits_{\Omega\ni x\to \partial \Omega}V^*(\delta_\Omega(x))u(x)=\infty.  & 
\end{cases}    
\end{equation}
where $L=-\phi(-\Delta)$ and \(V^*\) is a renewal function associated with the ladder-height process of a conjugate subordinator with the Laplace exponent $\phi^*$, for details see \eqref{eq:phistar} and \eqref{eq:Vstar}. We discuss how the blow-up condition in \eqref{eq:problem} reflects a beyond-moderate explosion at the boundary. The key novelty of this paper lies in extending the existence of large solutions to this more general nonlocal setting, under a nonlocal analogue of the Keller-Osserman condition that serves as a sufficient condition for the existence of such solutions. We also present the appropriate notion of the corresponding Keller-Osserman type condition, which, unlike \eqref{eq:KO_classical}, cannot be stated purely as an integrability condition, see \eqref{eq:KO_refined}. 

This paper is organized as follows. In Section 2, we provide the necessary background on the class of subordinators $\phi$ satisfying a certain scaling condition \eqref{eq:H} and the operator $L$, as well as basic notions from the corresponding potential theory, such as the renewal function $V$. In the same way as in \cite{benchrouda}, the probabilistic aspects of these notions and techniques are presented only for better interpretation. Section 3 is devoted to the formulation of the Keller-Osserman type condition in the nonlocal setting and the construction of a suitable supersolution, based on \cite{abatangelo2}. In Section 4, we prove the existence of large solutions using a monotone approximation scheme. By connecting the notions of distributional and weak dual super- and sub-solutions, we obtain a very general distributional version of the comparison principle. Section 5 discusses connections to Kato-type inequalities and explores some consequences of this result, as well as important open problems.

Finally, a brief comment on the less standard notation in the paper. For positive functions $f$ and $g$, we write $f\lesssim g$ ($f\gtrsim g$, $f\asymp g$) if there exists a constant $c>0$ such that $f\le c\, g$ ($f\ge c \,g$, $c^{-1} g\le f \le c\,g$).

\section{Preliminaries} \label{sec:prelim}
Let $\phi:[0,\infty)\to[0,\infty)$ be a complete Bernstein function with zero drift. This means that the function $\phi$ can be represented as
\[
\phi(\lambda)=\int_0^\infty \left(1-e^{-\lambda s}\right)\nu(s)\ds,
\]
where the corresponding L\' evy density $\nu$ is a completely monotone function. We refer to \cite{bernstein} for a comprehensive overview of the theory. In probability theory, Bernstein functions are associated with subordinators, i.e.~nonnegative L\' evy processes, via their Laplace exponent. Let $S=(S_t:t\ge0)$ be the subordinator with the Laplace exponent $\phi$,~i.e.
\[
\E[e^{-\lambda S_t}]=e^{-t\phi(\lambda)}.
\]
We are interested in operators associated with generators of subordinate Brownian motions in $\R^d$, $d\ge 2$. Let $B=(B_t:t\ge 0)$ be a Brownian motion on $\R^d$ independent of $S$ and let $X=(X_t:t\ge 0)$ be the corresponding subordinate Brownian motion, $X_t=B_{S_t}$. The process $X$ is a L\' evy process with the characteristic exponent $\xi\mapsto\phi(|\xi|^2)$, $\xi\in\R^d$, and the corresponding infinitesimal generator $L$ of the form
\begin{equation}\label{eq:L1}
Lu(x)=\text{P.V.}\int_{\R^d}(u(y)-u(x))j(|y-x|)\dy, \quad u\in C_c^2(\R^d),
\end{equation}
\cite[Theorem 31.5]{sato}. Here $j:(0,\infty)\to(0,\infty)$ is the radial L\' evy density of $X$ given by
\begin{equation}\label{eq:kernel}
j(r)=\int_0^\infty \frac{1}{(4\pi s)^{d/2}}e^{\frac{-r^2}{4s}}\nu(s)\ds,
\end{equation}
which is continuous and decreasing. Throughout the paper, we assume that $\phi$ satisfies the following scaling property at infinity,
\begin{equation}\label{eq:H}
   a_1\lambda^{\delta_1}\leq\frac{\phi(\lambda t)}{\phi(t)}\leq a_2\lambda^{\delta_2},\quad \lambda\geq 1, t\geq 1, 
\end{equation}
for some constants $a_1,a_2>0$ and $0<\delta_1\leq \delta_2<1$. As a simple consequence of \eqref{eq:H} we get that
 \begin{equation}\label{eq:phi1}
 \phi'(t)\asymp \frac{\phi(t)}{t},\ t\geq 1,
 \end{equation}
and by \cite[Theorem 2.3]{ubHp},
\begin{equation}\label{eq:j}
j(r)\asymp \frac{\phi(r^{-2})}{r^d},\ r\leq 1.
\end{equation}
For a bounded open set $U\subset \R^d$, define the first exit time from $U$ of the process $X$ as $\tau_U=\inf\{t\ge 0:\, X_t\not\in U\}$ and denote by $X^U$ the corresponding killed process. The transition densities $p^U$ of $X^U$ and $p$ of $X$ are related by the so-called Hunt formula
\begin{align*}\label{eq:p_D Hunt}
		p^U(t,x,y)=p(t,x,y)-\E_x[p(t-\tau_U,X_{\tau_U},y)\1_{\{t\ge\tau_U\}}], \qquad t>0, \, x,y\in\R^d.
	\end{align*}
Denote by $G_U(x,y)$ the Green function of $X^U$, defined by
\[
G_U(x,y)=\int_0^\infty p^U(t,x,y)\, dt, \qquad x,y\in\R^d,
\]
and by $G_Uf$ the corresponding Green potential of a function $f$,
\[
G_Uf(x)=\int_U f(y)G_U(x,y)\, dy,\qquad x\in\R^d.
\]
For $x \in U$, the probability measure $\omega^x_U(dz) := \Pp_x(X_{\tau_U}\in dz)$ is called the harmonic measure of $U$ with respect to the process $X$. For a function $u$ on $U^c$ let 
\begin{equation}\label{eq:harmonic_potential}
P_Uu(x) := \E_x[u(X_{\tau_U})] =\int_{\R^d}u(y)\omega^x_U(dy),    
\end{equation}
whenever the integral makes sense. The Poisson kernel of $X^U$, defined by
\[
P_U(x,z)=\int_U G_U(x,y)j(|y-z|)\, dy,\qquad x\in U,\, z\in \overline{U}^c,
\]
is the density of the restriction of the harmonic measure to $\overline{U}^c$. Furthermore, if $\partial U$ is regular enough, e.g. if $U$ is a Lipschitz set, then
$\omega_U^x(dz)=P_U(x,z)dz$ on ${U}^c$.
Fix $x_0\in U$. For $x\in U$ and $z\in\partial U$ by \cite[
Lemma 3.4, Theorem 1.1]{KSV18a} there exists the limit
\[
M_U(x, z) := \lim_{U\ni y\to z}\frac{G_U(x,y)}{G_U(x_0,y)},
\]
which is called the Martin kernel of $U$ with respect to $X$. The Martin potential $M_Uh$ of a continuous function $h$ on $\partial U$ is defined accordingly. The aforementioned potentials play a crucial role in defining the weak dual solutions to (semi-)linear problems
\begin{equation}\label{e:semilinear_moderate}
	\begin{array}{rcll}
		-Lu&=& -f(u)& \quad \text{in } \Omega\\
		u&=&g& \quad \text{in }\overline{\Omega}^c\\
		W_\Omega u&=&h&\quad \text{on }\partial \Omega,
	\end{array}
\end{equation}
which form the class of moderate solutions to \eqref{e:semilinear}. Here, $W_\Omega$ is the appropriate boundary trace operator defined for general open sets, for details we refer to \cite{BVW} and \cite{Bio}. A function $u$ is a weak dual solutions to \eqref{e:semilinear_moderate} if it satisfies the relation
\[
u(x)=-G_\Omega f(u)(x)+P_\Omega g(x)+M_\Omega h(x), \qquad x\in\R^d.
\]
For details on this class of problems and the boundary operator $W_\Omega$, we refer to \cite{BJK}, \cite{Bio}, and \cite{BVW}.

In this paper, the boundary conditions and the beyond-moderate boundary blow-up of solutions will be determined in terms of the special renewal function. When $\Omega$ is a $C^{1,1}$ open set, the connection between the boundary operator $W_\Omega$ and the pointwise boundary condition in terms of the renewal function has been explored in detail in \cite[Subsection 4.7, (4.18), (4.19), (4.25)]{BVW}. Here we give the definition of the renewal function and a brief overview of its properties, for details we refer to \cite[Section VI]{bertoin} and \cite{KSV12b}. Let $Z=(Z_t)_{t\ge 0}$ be a one-dimensional subordinate Brownian motion with the characteristic exponent $\phi(\theta^2)$, $\theta\in \R$. The renewal function $V$ of the ascending ladder height process $H=(H_t)_{t\ge 0}$ of $Z$, is defined as
$$
V(t):=\int_0^{\infty}\mathbf P(H_s\le t)\, ds, \quad t\in \R.
$$
For the definition and properties of the ladder height process $H$, we refer to \cite[Section VI.1]{bertoin}. From the definition of the renewal function $V$ it easily follows that $V(t)=0$ for $t<0$, $V(0)=0$, $V(\infty)=\infty$, and $V$ is strictly increasing. In the case of the isotropic $\alpha$-stable process we have that $V(t)=t^{\alpha/2}$. In the general case, the function $V$ is not known explicitly, but rather determined by its Laplace transform. Nevertheless, under \eqref{eq:H}, it is known that $V$ is a $C^2$ function with the following properties:
\begin{align}
V(t) &\asymp \Phi(t):=\phi(t^{ -2})^{-1/2}\, ,\quad 0<t\leq 1,\label{e:V-phi}\\
V'(t)&\lesssim\frac{V(t)}{t}\, ,\quad 0<t\leq 1,\label{eq:H1}\\
|V''(t)|&\lesssim\frac{V'(t)}{t}\, ,\quad 0<t\leq 1.\label{eq:H2}
\end{align}
For details see \cite[Lemma 2.4, Lemma 2.5]{KKLL} and \cite{KSV12b}. Next, let $\phi^*$ be the conjugate of the complete Bernstein function $\phi$, i.e. a Bernstein function given by
\begin{equation}\label{eq:phistar}
 \phi^*(s)=\frac s{\phi(s)},\ \ s>0.
\end{equation}
The function $\phi^*$ is the Laplace exponent of the subordinator conjugate to $S$, and satisfies the weak scaling property \eqref{eq:H}, with different coefficients. Let $V^*$ denote the renewal function of the corresponding ascending ladder height process. Then, as an analogue of \eqref{e:V-phi} we get that
\begin{equation}\label{eq:Vstar}
V^*(s)\asymp\frac1{\sqrt{\phi^*(s^{-2})}}\asymp\frac{s}{V(s)},\ s>0.
\end{equation}
Since the nonnegative solution $u$ to \eqref{e:semilinear_moderate} with $g\equiv0$ is bounded from above by the harmonic function $M_\Omega h$, by \cite[(4.18)]{BVW} we conclude that
\begin{equation*}\label{e:blowup_moderate}
u(x)\leq M_\Omega h(x) \lesssim \frac{\delta_\Omega(x)}{V(\delta_\Omega(x))}\asymp \frac{1}{V^*(\delta_\Omega(x))}, \quad x\in \Omega.
\end{equation*}
The sharpness of this blow-up rate follows from the estimate
\begin{equation}\label{e:martin_estimate}
M_\Omega 1 (x) \asymp V^*(\delta_\Omega(x))  , \quad x\in \Omega,
\end{equation}
see \cite[(4.16)]{BVW} (also \cite[(4.19)]{BVW}). Moreover, by the Martin representation, \cite[Proposition 5.11]{Bio}, every nonnegative function $u$ harmonic on $\Omega$ for $X$ can be represented as an appropriate Martin potential. This implies that the blow-up rate of moderate solutions is at most $V^*(\delta_{\Omega})^{-1}$, which motivates the boundary condition in \eqref{eq:problem}.

Let us now turn to the nonlinearity $f$, which is, following \cite{abatangelo2}, modelled as a generalization of the power function $f(t)=t^p$, $p>1$. Let $f:\R\to[0,\infty)$ be an increasing $C^1(\R)$ function with $f(0)=0$ such that
\begin{equation}\label{eq:f1}
(1+m) f(t)\leq tf'(t)\leq (1+M) f(t),\ t\in\R,
\end{equation} 
for some constants $0<m\leq M$. Denote by $F:(0,\infty)\to(0,\infty)$ the antiderivative of $f$, i.e.
\[
F(t)=\int_0^t f(s)\ds,\ t>0.
\]
 By \eqref{eq:f1} the function $t\mapsto F(t)^{-1/2}$ is integrable at infinity.  Indeed, by integrating \eqref{eq:f1} we get that $f(t)\geq f(1) t^{1+m}$ and therefore $F(t)\gtrsim t^{2+m}$ for $t\geq 1$.
Thus one can define a function $\varphi:(0,\infty)\to (0,\infty)$ as
\begin{equation}
\varphi(t)=\int_t^\infty\frac{\ds}{\sqrt{F(s)}}, \ t>0.
\end{equation}

\begin{remark}\label{rem:f}
\begin{enumerate}
\item The function $\varphi$ is monotone decreasing and
\[
\lim\limits_{t\downarrow 0}\varphi(t)=\infty,\ \lim\limits_{t\uparrow\infty}\varphi(t)=0.
\]
\item Using \eqref{eq:f1} and the definition of $\varphi$ one can show that
\begin{equation}\label{eq:varphi1}
\frac m 2 \frac{\varphi(t)}{t}\leq |\varphi'(t)|\leq \frac M 2 \frac{\varphi(t)}{t},\ t>0,
\end{equation}
and
\begin{equation}\label{eq:f2}
\sqrt{\frac{t}{f(t)}}\asymp\varphi(t),\ t>0, 
\end{equation}
for details see \cite[Remark 1.1]{abatangelo2}.
\end{enumerate}
\end{remark}
Denote by $\psi$ the inverse function of $\varphi$, which will appear in the construction of the supersolution. By Remark \ref{rem:f}, $\psi$ is decreasing and satisfies
\[
\lim\limits_{t\downarrow 0}\psi(t)=\infty,\ \lim\limits_{t\uparrow\infty}\psi(t)=0.
\]
 Furthermore, by \eqref{eq:f1} it follows that 
\begin{equation}\label{eq:psi1}
	\frac{t^2\psi''(t)}{\psi(t)}\asymp \frac{t^2\psi'(t)^2}{\psi(t)^2}\asymp 1.
\end{equation}
Note that \eqref{eq:varphi1} written in terms of $\psi$ is of the form
\begin{equation}\label{eq:psi2}
\frac 2 M \frac{\psi(t)}{t}\leq |\psi'(t)|\leq \frac 2 m\frac{\psi(t)}{t},\ t>0,
\end{equation}
which in turn implies that the function $t\mapsto \psi(t)t^{\frac 2 m}$ is non-decreasing and the function $t\mapsto \psi(t)t^{\frac 2 M}$ is non-increasing.

Throughout the paper, we assume $d\ge 2$ and that $\Omega\subset\R^d$ is a bounded open $C^{1,1}$ set with characteristics $(R,\Lambda)$. Denote by $\delta_{\Omega}(x):=\dist(x,\partial\Omega)$ the distance of a point $x\in\Omega$ to the boundary $\partial\Omega$ and for $\eta>0$ set $\Omega_\eta=\{x\in\Omega:\delta_{\Omega}(x)< \eta\}$. Since $\Omega$ is a $C^{1,1}$ set, by \cite[Theorem 5.4.3(i)]{DZ} it follows that $\delta_\Omega$ is a $C^{1,1}$ on an open neighbourhood of each point at the boundary $\partial\Omega$. Since $\Omega$ is bounded, this means that there exists $\eta_0>0$ such that $\delta_\Omega\in C^{1,1}(\Omega_{\eta_0})$.

The results of this paper focus on the construction of a distributional solution to the semilinear problem \eqref{eq:problem}, i.e. the existence of a function $u\in L^1(\Omega)$ such that for every $\xi\in C_c^\infty(\Omega)$
\[
\int_{\Omega} u L\xi=\int_{\Omega} f(u)\xi, 
\]
and $\lim\limits_{\Omega\ni x\to\partial\Omega} u(x)V^*(\delta_\Omega(x))=\infty$. Recall that this limiting boundary condition implies that $u$ is not harmonically dominated, i.e.~does not fall into the class of the so-called moderate solutions investigated in \cite{BVW}. Since the solution $u$ to \eqref{eq:problem} exhibits this stronger singularity at $\partial \Omega$ than the moderate solutions to \eqref{e:semilinear},  we call $u$ a large solution to \eqref{e:semilinear}.

\section{A Keller-Osserman condition for nonlocal operators}\label{sec-KO}

The key ingredient in the proof of the existence of a solution to \eqref{eq:problem} is the construction of an integrable pointwise supersolution, under the so-called nonlocal Keller-Osserman condition.
The appropriate boundary behaviour of the supersolution is governed by the function $U:\R^d\to(0,\infty)$ defined by
\[
	U(x)=\psi\left(V(\delta_\Omega(x))\right), \qquad x\in\Omega.
\] 
First, we introduce a weaker form of the nonlocal Keller-Osserman condition, which assures the integrability condition on the supersolution, needed to construct the sequence of moderate weak dual solutions to the approximating problems \eqref{e:uk}. In what follows we use the abbreviation $\delta=\delta_\Omega$.

\begin{lemma}\label{lem:integrability}
 The function $U$ is in $L^1(\Omega)$ if and only if  the Keller-Osserman-type condition
\begin{equation}\label{eq:KO}
\int_1^\infty \frac{\dt}{\sqrt{\phi^{-1}(\varphi(t)^{-2})}}<\infty
\end{equation}
is satisfied.
\end{lemma} 
\proof
Since  $\psi$ and $V$ are continuous on $(0,\infty)$ and $\delta\in C(\Omega)$, it follows that $U\in L^1_{loc}(\Omega)$. Therefore, it is enough to show that \eqref{eq:KO} holds if and only if $U\in L^1(\Omega_{\eta_0})$. Denote by $\mathcal H^{d-1}$ the $(d-1)$-dimensional Hausdorff measure on $\partial\Omega$. Recall the definition of $\Phi$ from \eqref{e:V-phi}, as well as the relation $V\asymp \Phi$. By applying the coarea formula, see e.g. \cite[Theorem 3.5.6]{DZ}, we arrive to
\begin{align*}
\int_{\Omega_{\eta_0}}U(x)\dx&=\int_0^{\eta_0} \dt\int_{\{x\in\Omega:\delta(x)=t\}}\psi(V(t))d\mathcal H^{d-1}(\dx)\overset{\eqref{e:V-phi}}{\asymp} \int_0^{\eta_0}\psi(\Phi(t))\dt\\
&= \int_{\psi(\Phi(\eta_0))}^\infty s(\Phi^{-1})'(\varphi(s))|\varphi'(s)|\ds\overset{\eqref{eq:varphi1}}{\asymp} \int_{\psi(\Phi(\eta_0))}^\infty (\Phi^{-1})'(\varphi(s))\varphi(s)\ds\\
&=\int_{\psi(\Phi(\eta_0))}^\infty \frac{\varphi(s)}{\Phi'(\Phi^{-1}(\varphi(s)))}\ds\overset{\eqref{eq:phi1}}{\asymp}\int_{\psi(\Phi(\eta_0))}^\infty \frac{\ds}{\sqrt{\phi^{-1}(\varphi(s)^{-2}))}},
\end{align*}
where in the second line we used the substitution $s=\psi(\Phi(t)))$ and the fact that the functions $\psi\circ\Phi$ and $\varphi$ are decreasing.
\qed

In order for $U$ to be a pointwise supersolution near the boundary of $\Omega$, we need a slightly stronger version of \eqref{eq:KO}, given by the inequality \eqref{eq:KO_refined}. Note that when $f(t)=t^p$, $p>0$, and $\phi(t)=t^{\alpha/2}$ the Keller Osserman-type conditions \eqref{eq:KO_refined} and \eqref{eq:KO} are both equivalent to $p>\alpha+1$. For the class of functions $f$ treated in this paper, as well as in \cite{abatangelo2}, the condition \eqref{eq:KO} is not enough to construct the supersolution, even in the fractional setting. This is because the comparability in \eqref{eq:psi2} is not enough to deduce \eqref{eq:KO_refined} from \eqref{eq:KO}, note the change of sign in the line before \cite[(1.29)]{abatangelo2}. See also the following Remark \ref{rem:KO_refined}.

\begin{prop}\label{prop:supersolution}
Assume that for $R>0$ there exists a constant $C_1>0$ such that for every $r\geq R$
\begin{equation}\label{eq:KO_refined}
\int_r^\infty \frac{\dt}{\sqrt{\phi^{-1}(\varphi(t)^{-2})}}\leq C_1 \frac{r}{\sqrt{\phi^{-1}(\varphi(r)^{-2})}}.
\end{equation}
Then there exist constants $\eta_1\in(0,\eta_0)$ and $C_2>0$ such that 
\[
LU(x)\leq C_2 f(U(x)),\ x\in\Omega_{\eta_1}.
\]
\end{prop}
\begin{remark}\label{rem:KO_refined}
\begin{enumerate}
 \item Note that the stronger condition \eqref{eq:KO_refined} follows from \eqref{eq:KO} and \eqref{eq:psi2} when 
 \begin{equation}\label{eq:kato_m}
\int_0^\varepsilon \phi(t^{-2})^{\frac 1{m}}\dt \lesssim \varepsilon  \phi(\varepsilon^{-2})^{\frac 1{m}}.   
 \end{equation}
 To see this, substitute $t$ in \eqref{eq:KO_refined} with $u=(\phi^{-1}(\varphi(t)^{-2}))^{-1/2}$ and denote $\widetilde{r}=(\phi^{-1}(\varphi(r)^{-2}))^{-1/2}$. It follows that
\begin{align*}
\int_r^\infty \frac{\dt}{\sqrt{\phi^{-1}(\varphi(t)^{-2})}}&= \int_0^{\widetilde{r}} u|\psi'(\phi(u^{-2})^{-1/2})|\frac{\phi'(u^{-2})}{\phi(u^{-2})^{3/2}}u^{-3}du\\
&\overset{\eqref{eq:phi1}}{\asymp} \int_0^{\widetilde{r}} |\psi'(\phi(u^{-2})^{-1/2})|\frac{1}{\phi(u^{-2})^{1/2}}du\\
&\overset{\eqref{eq:psi2}}{\asymp} \int_0^{\widetilde{r}} \psi(\phi(u^{-2})^{-1/2})du.
\end{align*}
Since \eqref{eq:psi2} implies that the function $t\mapsto\psi(t)t^{2/m}$ is non-decreasing, we get that 
\begin{align*}
\int_r^\infty \frac{\dt}{\sqrt{\phi^{-1}(\varphi(t)^{-2})}}&\leq \psi(\phi(\widetilde{r}^{-2})^{-1/2})\phi(\widetilde{r}^{-2})^{-1/m} \int_0^{\widetilde{r}} \phi(u^{-2})^{1/m}du\\
&\lesssim \psi(\phi(\widetilde{r}^{-2})^{-1/2})\widetilde{r}=\frac{r}{\sqrt{\phi^{-1}(\varphi(r)^{-2})}}.
\end{align*}
 In the standard $\alpha$-stable case \eqref{eq:kato_m} is equivalent to $m>\alpha$.
 
\item By applying the same calculation as in the proof of Lemma \ref{lem:integrability}, condition \eqref{eq:KO_refined} is equivalent to 
\begin{equation}\label{eq:psi-V-integral}
 \int_0^{\varepsilon}\psi(V(t))\dt\asymp\int_{\psi(V(\varepsilon))}^\infty \frac{\ds}{\sqrt{\phi^{-1}(\varphi(s)^{-2}))}}\lesssim \frac{\psi(V(\varepsilon))}{\sqrt{\phi^{-1}(\varphi(\psi(V(\varepsilon)))^{-2})}}\asymp \varepsilon\psi(V(\varepsilon)).
\end{equation}
\end{enumerate}
\end{remark}

\proofof \emph{Proposition \ref{prop:supersolution}}. 
Fix $\eta\in(0,\eta_0/2)$ and take $x\in\Omega_\eta$. Since $U=0$ on $\Omega^c$ and $U\ge 0$ on $\Omega$, we have that
\begin{align*}
LU(x)&=\text{P.V.}\int_\Omega (U(y)-U(x))j(|y-x|)\dy + \int_{\Omega^c} (U(y)-U(x))j(|y-x|)\dy\\
&\le \text{P.V.}\int_\Omega (U(y)-U(x))j(|y-x|)\dy=:I.
\end{align*}
In order to estimate $I$, we decompose the integral into three parts
\begin{align*}
I=&\int_{\Omega_1} (U(y)-U(x))j(|y-x|)\dy+\text{P.V.}\int_{\Omega_2} (U(y)-U(x))j(|y-x|)\dy\\
&+\int_{\Omega_3} (U(y)-U(x))j(|y-x|)\dy=:J_1+J_2+J_3,
\end{align*}
where $\Omega=\Omega_1\cup\Omega_2\cup\Omega_3$ such that
\begin{align*}
&\Omega_1=\left\{y\in\Omega: \delta(y)>\tfrac{3}{2}\delta(x)\right\}\\
&\Omega_2=\left\{y\in\Omega: \tfrac{1}{2}\delta(x)\leq\delta(y)\leq\tfrac{3}{2}\delta(x)\right\}\\
&\Omega_3=\left\{y\in\Omega: \delta(y)<\tfrac{1}{2}\delta(x)\right\}.
\end{align*}
Since the function $\psi\circ V$ is decreasing, it follows that $J_1\leq 0$. For $J_2$ we split the integral into two parts
\begin{align*}
J_2&=\text{P.V.}\int_{B(x,\delta(x)/2)} (U(y)-U(x))j(|y-x|)\dy\\
&+\int_{\Omega_2\setminus B(x,\delta(x)/2)} (U(y)-U(x))j(|y-x|)\dy=:J_2^1+J_2^2.
\end{align*}
In order to get an upper bound for $J_2^1$ first note that $U\in C^{1,1}(\Omega_{\eta_0})$ and that $\nabla U=\psi'(V(\delta))V'(\delta)\nabla \delta$. Furthermore, \eqref{eq:H1}, \eqref{eq:H2}, \eqref{eq:psi1} and \eqref{eq:psi2} imply that for every $z\in B(x,\delta(x)/2)\subset\Omega_{\eta_0}$ we have
\begin{align}\label{e:gradient_U}
    |\nabla U(z)-\nabla U(x)|&\lesssim |\psi'(V(\delta(x)))V'(\delta(x))||\nabla \delta(z)-\nabla \delta(x)|\nonumber\\
    &\lesssim \big(|\psi'(V(\delta(z)))-\psi'(V(\delta(x)))|V'(\delta(x))|\nabla\delta(z)|\nonumber\\
    &\ +|\psi'(V(\delta(x)))|V'(\delta(z))-V'(\delta(x))||\nabla\delta(z)|\nonumber\\
    &\ +|\psi'(V(\delta(x)))V'(\delta(x))|\nabla\delta(z)-\nabla\delta(x)| \big)|z-x| \nonumber\\
    &\lesssim \big(\psi''(V(\delta(x)))V'(\delta(x))+|\psi'(V(\delta(x)))V''(\delta(x))|\nonumber\\
    &\ +|\psi'(V(\delta(x)))|V'(\delta(x)) \big)|z-x| \nonumber\\
   &\lesssim \frac{\psi(V(\delta(x)))}{\delta(x)^2}|z-x| .
\end{align}
Therefore,
\begin{align}
|J_2^1|&=\left|\int_{B(x,\delta(x)/2)} (U(y)-U(x)-\nabla U(x)(y-x))j(|y-x|)\dy\right|\nonumber\\
&\le\int_{B(x,\delta(x)/2)} \left(\int_0^1|\nabla U(x+t(y-x))-\nabla U(x)|\dt\right)|y-x| j(|y-x|)\dy\nonumber\\
&\overset{\eqref{e:gradient_U}}{\lesssim}\frac{\psi(V(\delta(x)))}{\delta(x)^2}\int_{B(x,\delta(x)/2)} |y-x|^2j(|y-x|)\dy\nonumber\\
&\overset{\eqref{eq:j}}{\lesssim} \frac{\psi(V(\delta(x)))}{\delta(x)^2}\int_{B(x,\delta(x)/2)} \frac{\phi(|y-x|^{-2})}{|y-x|^{d-2}}\dy\nonumber \\
&\overset{\eqref{eq:H}}{\lesssim} \frac{\psi(V(\delta(x)))}{\delta(x)^2}\phi(\delta(x)^{-2})\delta(x)^2=\psi(V(\delta(x)))\phi(\delta(x)^{-2}).\label{eq:J21}
\end{align}
Similarly, since $\psi\circ V$ is decreasing it follows from \eqref{eq:j} and \eqref{eq:H} that for some $R>0$
\begin{align*}
J_2^2&\lesssim\psi(V(\delta(x)/2))\int_{\delta(x)/2}^{R} j(r)r^{d-1}\dr\lesssim \psi(V(\delta(x)))\phi(\delta(x)^{-2}),
\end{align*}
so by \eqref{e:V-phi} we get that
\begin{equation}
J_2^2\lesssim \psi(V(\delta(x)))V(\delta(x))^{-2}.\label{eq:J22}
\end{equation}
Let $Q\in\partial\Omega$ be the projection of $x$ on $\partial\Omega$, that is $x=Q+\delta(x)\nabla\delta(x)$, and $\gamma=\gamma_Q$ the $C^{1,1}$ function such that
\begin{equation*}
B(Q,R) \cap\Omega = \{y = (\widetilde y,y_d) \in B(Q,R) \text{ in } CS_Q : y_d > \gamma(\widetilde y)\},
\end{equation*}
where $CS_Q$ denotes the coordinate system with $Q=0$ and $\nabla\delta(x)=e_d$. Set 
\begin{equation*}
\omega=\{y\in\Omega:y=Q_y+\delta(y)\nabla\delta(y),\ Q_y\in B(0,R) \cap\partial \Omega\}
\end{equation*}
and
\begin{equation*}
J_3=\int_{\Omega_3\setminus\omega } (U(y)-U(x))j(|y-x|)\dy+\int_{\Omega_3\cap \omega} (U(y)-U(x))j(|y-x|)\dy=:J_3^1+J_3^2.
\end{equation*}
First note that for $y\in\Omega_3\setminus \omega$
\begin{equation*}
y=Q_y+\delta(y)\nabla\delta(y),\ Q_y\not\in B(0,R)
\end{equation*}
and therefore 
\begin{align*}
|y-x|\geq |Q_y|-\delta(y)-\delta(x)\geq R-\frac 5 2 \delta(x)> R-\frac 5 2 \eta>0,
\end{align*}
by choosing $\eta$ small enough. This and nonnegativity of $U$ imply that
\begin{align}\label{eq:J31}
J_3^1\lesssim \int_{\Omega_3\setminus \omega} U(y)\dy\leq||U||_{L^1(\Omega_{\eta_0})}. 
\end{align}
On the other hand, for $y\in\Omega_3\cap\omega$ we have that $y=(\widetilde{Q_y},\gamma(\widetilde{Q_y}))+\delta(y)\frac{(-\nabla\gamma(\widetilde{Q_y}),1)}{\sqrt{|\nabla\gamma(\widetilde{Q_y})|^2+1}}$. Since $\gamma\in C^{2}$, we get that
\begin{equation*}
 |y_d-\delta(y)|=\left|\gamma(\widetilde{Q_y})+\frac{\delta(y)}{\sqrt{|\nabla\gamma(\widetilde{Q_y})|^2+1}}-\delta(y) \right|\leq \Lambda\eta|\widetilde{Q_y}|^2+\Lambda|\widetilde{Q_y}|
\end{equation*}
and 
\begin{align*}
 |\overline{y}|^2&=\left|\widetilde{Q_y}-\frac{\delta(y)\nabla\gamma(\widetilde{Q_y})}{\sqrt{|\nabla\gamma(\widetilde{Q_y})|^2+1}} \right|^2\geq \left|\widetilde{Q_y}\right|^2-2\left|\widetilde{Q_y}\right|\delta(y)|\nabla\gamma(\widetilde{Q_y})|+\frac{\delta(y)^2|\nabla\gamma(\widetilde{Q_y})|^2}{{|\nabla\gamma(\widetilde{Q_y})|^2+1}}\\
 &\geq \left|\widetilde{Q_y}\right|^2-2\Lambda\eta\left|\widetilde{Q_y}\right|^2\geq \tfrac 1 2|\widetilde{Q_y}|^2
\end{align*}
for $\eta\leq\frac{1}{4\Lambda}$. This implies that
\begin{align*}
 |y-x|^2&=|(\delta(x)-\delta(y))e_d+(\delta(y)-y_d)e_d-(\overline{y},0)|^2\\
 &\geq (\delta(x)-\delta(y))^2-2|\delta(x)-\delta(y)||\delta(y)-y_d|+(\delta(y)-y_d)^2+|\overline{y}|^2 \\
 &\geq  (\delta(x)-\delta(y))^2-2\eta\Lambda|\delta(x)-\delta(y)||\widetilde{Q_y}|^2-2\Lambda|\delta(x)-\delta(y)||\widetilde{Q_y}|+\tfrac 1 2|\widetilde{Q_y}|^2\\
 &\geq  (\delta(x)-\delta(y))^2-2\Lambda\eta^2R^2-2\Lambda\eta R+\tfrac 1 2|\widetilde{Q_y}|^2.
\end{align*}
By choosing $\eta$ small enough it follows that there exists a constant $C=C(\Lambda,R)$ such that
\begin{equation}\label{eq:Qy}
 |y-x|^2\geq C((\delta(x)-\delta(y))^2+|\widetilde{Q_y}|^2).
\end{equation}
Therefore, by \eqref{eq:KO_refined} and Remark \ref{rem:KO_refined}(ii) we have that
\begin{align}
 J_3^2&\leq\int_{\Omega_3\cap w}U(y)j(|y-x|)\dy\overset{\eqref{eq:j}}{\lesssim}\int_{\Omega_3\cap w}U(y)\frac{\phi(|y-x|^{-2})}{|y-x|^d}\dy\nonumber\\
 &\overset{\eqref{eq:Qy}}{\lesssim}\int_{\Omega_3\cap w}\psi(V(\delta(y)))\frac{\phi(((\delta(x)-\delta(y))^2+|\widetilde{Q_y}|^2)^{-1})}{((\delta(x)-\delta(y))^2+|\widetilde{Q_y}|^2)^{d/2}}\dy\nonumber\\
 &\asymp\int_{0}^{\delta(x)/2}\int_0^R\psi(V(t))\frac{\phi(((\delta(x)-t)^2+s^2)^{-1})}{((\delta(x)-t)^2+s^2)^{d/2}}s^{d-2}\ds\dt\nonumber\\
 &\lesssim\int_{0}^{\delta(x)/2}\psi(V(t))\frac{\phi((\delta(x)-t)^{-2})}{|\delta(x)-t|}\dt\asymp \frac{\phi(\delta(x)^{-2})}{\delta(x)}\int_{0}^{\delta(x)/2}\psi(V(t))\dt\nonumber\\
 &\overset{\eqref{eq:psi-V-integral}}{\lesssim} \phi(\delta(x)^{-2})\psi(V(\delta(x)))\overset{\eqref{e:V-phi}}{\asymp} \psi(V(\delta(x)))V(\delta(x))^{-2}.\label{eq:J32}
\end{align}
Inequalities \eqref{eq:J21}, \eqref{eq:J22}, \eqref{eq:J31} and \eqref{eq:J32} now imply that $I\lesssim \psi(V(\delta(x)))V(\delta(x))^{-2}$.
Finally, by applying \eqref{eq:f2} we get that
\begin{equation*}
\frac{\psi(V(\delta(x)))}{V(\delta(x))^2}=\frac{U(x)}{\varphi(U(x))^2}\asymp f(U(x)).
\end{equation*}
\qed

\begin{lemma}\label{lem:boundary_rate}
 The function $U$ satisfies the boundary growth condition 
\begin{equation*}
 \lim_{x\to\partial\Omega} V^*(\delta(x))U(x)=\infty
\end{equation*}
if and only if 
\begin{equation}\label{eq:growth}
\lim_{s\to\infty} V^{-1}\left(\sqrt{\frac{s}{f(s)}}\right)\sqrt{sf(s)}=\infty.
\end{equation}

\end{lemma}
\proof
By substituting $v=V(\delta(x))$ we get
\begin{align*}
 \lim_{x\to\partial\Omega} V^*(\delta(x))U(x)&=\lim_{v\to 0}\frac{V^{-1}(v)}{v}\psi(v)=\lim_{s\to \infty}\frac{s}{\varphi(s)}V^{-1}(\varphi(s))\\
 &\overset{\eqref{eq:f2}}{\asymp}\lim_{s\to \infty}\sqrt{sf(s)}V^{-1}\left(\sqrt{\frac{s}{f(s)}}\right)=\infty.
 \end{align*}
\qed

\begin{lemma}\label{l:supersolution}
Let $f$ be an nondecreasing $C^1$ function such that $f(0)=0$ satisfying conditions \eqref{eq:f1}, \eqref{eq:KO_refined} and \eqref{eq:growth}. 
Then there exists a supersolution $\overline{u}:\R^n\to\R$ to \eqref{eq:problem} such that $\overline{u}\in C^{1,1}(\Omega)$.
\end{lemma}
\proof
We will construct a supersolution $\overline{u}$ of the form
\begin{equation*}
 \overline{u}(x)=aU(x)+bG_\Omega1(x), \ x\in\R^d,
\end{equation*}
for some constants $a,b>0$. Since $U\in C^{1,1}(\Omega)$,  \cite[Theorem 1.2]{GKL} implies that $\overline u\in C^{1,1}(\Omega)$. By Proposition \ref{prop:supersolution}, there exists a constant $c\ge 1$ such that for all $x\in\Omega_{\eta_1}$ we have that
\begin{align*}
-L\overline{u}(x)=&-aLU(x) -bLG_\Omega 1(x)\ge -acf(U(x))+b\ge -acf(U(x))\ge -acf(\tfrac 1 a \overline{u}(x)),
\end{align*}
where the last inequality follows from the fact that $G_\Omega 1\ge 0$ and the assumption that $f$ is nondecreasing. 
Set $a= c^{1/m}>1 $. 
Applying \eqref{eq:f1}, i.e. the fact that $t\mapsto f(t)t^{-1-m}$ is nondecreasing, we get that
$$
f(\overline{u}(x))=f\left(a\frac{1}{a}\overline{u}(x)\right)\ge a^{1+m}f\left(\frac{1}{a}\overline{u}(x)\right),
$$
and thus
$$
-L\overline{u}(x)\ge -ac a^{-1-m}f(\overline{u}(x)) =-a^{-m} cf(\overline{u}(x))=-f(\overline{u}(x)),\ x\in\Omega_{\eta_1}.
$$
Since $LU$ is bounded on $\Omega\setminus\Omega_{\eta_1}$ we can choose $b=a\sup_{y\in \Omega\setminus\Omega_{\eta_1}} |LU(y)|$ which implies that 
\[
-L\overline{u}(x)= -aLU(x)-bLG_\Omega 1(x)=-aLU(x)+b\ge 0\ge -f(\overline{u}(x)),\ x\in \Omega\setminus\Omega_{\eta_1}.
\]
\qed

\begin{remark}
Note that when $\Omega$ is a ball the supersolution $\overline{u}$ from Lemma \ref{l:supersolution} is a radial function. 
\end{remark}

Next, we want to connect the notions of distributional and probabilistic superharmonic (respectively, subharmonic) functions. The latter is defined in terms of the potential $P_U$ from \eqref{eq:harmonic_potential}, for $U\subset \R^d$ open and bounded. In the case of harmonic functions for integro-differential operators generating unimodal L\' evy processes, this connection has been made in \cite{GKL}. 

\begin{lemma}\label{l:superharmonic}
Let $h\in L^1(\R^d, (1\wedge j(|x|))dx)\cap C(\Omega)$ be such that for every $v\in C_c^\infty(\Omega)$, $v\ge 0$,
\begin{equation}\label{eq:h_assump}
\int_{\R^d} h(x) (-L v)(x)\, dx  \ge 0.
\end{equation}
Then for every relatively compact open set $A\subset \Omega$ and $x\in\R^d$
\[
h(x) \ge P_A h(x). 
\]
\end{lemma}

\proof Let $A\subset \Omega$ be a relatively compact open set. We first prove the claim for $h\in C^2(\Omega)\cap L^1(\R^d, (1\wedge j(|x|)dx))$. Therefore, $-Lh$ can be calculated pointwise, and the assumption \eqref{eq:h_assump} is equivalent to $-Lh(x)\ge 0$ for $x\in\Omega$.
Denote by $A^{\text{reg}}\subset A^c$ the set of all regular points for $A^c$, i.e. $x\in A^{\text{reg}}$ whenever $\mathbf P_x (\tau_A=0)=1$. Note that $(\overline{A})^c\subset A^{\text{reg}}$ and by \cite[Proposition II (3.3), p.80]{BG} $\partial A\setminus A^{\text{reg}}$ is a semi-polar set for $X$.  Since $X$ is symmetric and has a $\alpha$-potential density, by \cite[Proposition VI.4.10]{BG} semi-polar sets for $X$ are polar.  This implies that $\partial A\setminus A^{\text{reg}}$ is polar and therefore 
\begin{equation}\label{e:polar}
\mathbf P_x(X_{\tau_A}\in \partial A\setminus A^{\text{reg}})=0,\quad x\in\R^d.
\end{equation}
Modifying the approach in \cite[Lemma 3.2]{GKL}, set 
\begin{equation*}
\widetilde h(x)=\begin{cases}
P_{A}h(x), & x\in A,\\
h(x), & \text{otherwise}.
\end{cases}
\end{equation*}
Since $\widetilde h(x)=h(x)=P_A h(x)$ for all $x\in A^{\text{reg}}$, it follows from \eqref{e:polar} that for every $x\in A$
\begin{align*}
P_A \widetilde{h}(x)&=\mathbf E_x[\widetilde h(X_{\tau_A})]=\mathbf E_x[\widetilde h(X_{\tau_A});X_{\tau_A}\in A^{\text{reg}}]=\mathbf E_x[P_Ah(X_{\tau_A});X_{\tau_A}\in A^{\text{reg}}]\\
&=\mathbf E_x[P_Ah(X_{\tau_A})]=P_A h(x)=\widetilde h (x),
\end{align*}
so the function $\widetilde h$ is regular harmonic in $A$. Since 
$\widetilde h\in  L^1(\R^d,(1\wedge j(x))dx)$, by \cite[Lemma 2.2, Lemma 3.1]{GKL} $\widetilde h\in C^2(A)$ and $L \widetilde h(x)=0$, $x\in A$. Note that under \eqref{eq:H}, the L\' evy density $j$ from \eqref{eq:kernel} satisfies conditions \cite[(A), (1.2)]{GKL} assumed in \cite[Lemma 2.2]{GKL}, for details see \cite[Example 6.2]{GKL}. Set $u:=\widetilde h-h$. We have shown that $u\in C^2(A)$ and, by definition, $u\equiv 0$ on $A^c$. Since for every $x\in A$, $L u(x)=-L h(x)$, it follows that $Lu(x) \ge 0$ for $x\in A$. Assume that $\sup\{u(x):x\in A\}>0$. Since $u\in C_c(\Omega)$ and $u=0$ on $A^c$, the supremum is attained at some $x_0\in A$. It follows that
\begin{align*}
Lu(x_0)&=\int_{\R^d}(u(x)-u(x_0))j(|x-x_0|)\, dx=\\
&=\int_{A}(u(x)-u(x_0))j(|x-x_0|)\, dx-u(x_0)\int_{A^c}j(|x-x_0|)\, dx <0,
\end{align*}   
which contradicts the already established inequality $Lu(x_0)\ge 0$. Therefore, we get that $u\le 0$ on $A$, that is $h\ge P_Ah$ on $A$.  

For a general function $h\in L^1(\R^d, (1\wedge j(|x|)dx))\cap C(\Omega)$ satisfying \eqref{eq:h_assump} define $h_\varepsilon=\theta_\varepsilon\ast h$, where $\theta_\varepsilon$, $\varepsilon>0$, are standard mollifiers. Note that $h_\varepsilon$ satisfies \eqref{eq:h_assump} on $\Omega\setminus\Omega_\varepsilon$, since for $v\in C_c^\infty(\Omega\setminus\Omega_\varepsilon)$, $v\ge 0$, we have that $v_\varepsilon\in C_c^\infty(\Omega)$ and $v_\varepsilon\ge 0$, so 
\begin{align*}
\int_{\R^d} h_\varepsilon(x) L v(x)\, dx&= \int_{\rd} \int_{\rd} \theta_{\epsilon}(y) h(x-y) Lv(x) \, dx \, dy= \int_{\rd}   \int_{\rd} \theta_{\epsilon}(y)h(z) Lv(z+y) dz \, dy\\
&=\int_{\R^d} h(z) (\theta_\varepsilon\ast L v)(z)\, dz=\int_{\R^d} h(y) L v_\varepsilon(y)\, dy \le 0.
\end{align*}
Since $h_\varepsilon\in C^2(A)$ it follows from the first part of the proof that $h_\varepsilon\ge P_Ah_\varepsilon$ on $A$, for $\varepsilon>0$ such that $\overline{A}\subset \Omega\setminus\Omega_\varepsilon$. By \cite[Lemma 2.9]{GKL} we have that $h_\varepsilon\xrightarrow{\varepsilon\downarrow 0}h$ in $L^1(\R^d, (1\wedge j(|x|))dx)$, which together with continuity of $h$ on $\Omega$ implies that
\begin{align}\label{eq:instead_of_fatou}
P_A h(x)=\lim_{\varepsilon\downarrow 0} P_A h_\varepsilon(x)\le \lim_{\varepsilon\downarrow 0} h_\varepsilon(x)= h(x).
\end{align} 
To see that the first equality in \eqref{eq:instead_of_fatou} holds, note that for any relatively compact subset $U$ of $\Omega$ such that $\overline{A}\subset U$ we have 
\begin{align}\label{eq:integral}
   |P_A(h-h_\varepsilon)(x)|&\le \int_{U\setminus A}|h(z)-h_\varepsilon(z)|\omega_A(x,dz)+\int_{U^c}|h(z)-h_\varepsilon(z)|P_A(x,z)dz.    
\end{align}
Since $h\in C(\Omega)$, we have that $h_\varepsilon\xrightarrow{\varepsilon\downarrow 0}h$ pointwise on $\Omega$. Moreover, $|h|, |h_\varepsilon|\le M$ on $U$ for some $M>0$ and all $\varepsilon>0$ small enough. Therefore, by the dominated convergence theorem, the first integral in \eqref{eq:integral} converges to zero as $\varepsilon\to 0$. For the second integral in \eqref{eq:integral} note that
\begin{equation*}
P_A(x,z)\lesssim 1\wedge j(|z|), \ x\in A,\  z\in U^c,
\end{equation*}
so the desired result follows from $h_\varepsilon\xrightarrow{\varepsilon\downarrow 0}h$ in $L^1(\R^d, (1\wedge j(|x|))dx)$.
\qed

\begin{remark}\label{r:superharmonic}
\begin{enumerate}
    \item The analogous statement of the Lemma \ref{l:superharmonic} for a subharmonic function $h$ is obtained by applying Lemma \ref{l:superharmonic} to the function $-h$.
    \item The statement of Lemma \ref{l:superharmonic} for nonnegative or locally bounded supersolutions holds without the continuity assumption when $A$ is a Lipschitz set, or more generally, $A\subset\R^d$ such that $\omega_A(x,\partial A)=0$. Indeed, for the convergence of the first integral in \eqref{eq:instead_of_fatou}, one can apply Fatou's lemma in the first case or the dominated convergence theorem in the second, since $h_\varepsilon\xrightarrow{\varepsilon\to 0} h$ a.e.~on $\R^d$ and $\omega_A(x,dz)=P_A(x,z)dz$. 
\end{enumerate}
\end{remark}

\begin{lemma}\label{l:1}
If a nonnegative function $u\in L^1(\R^d, (1\wedge j(|x|))dx)\cap C(\Omega)$ is a distributional supersolution of 
\begin{equation}\label{e:11}
\begin{array}{rcll}
		-Lu&=& -f(u)& \quad \text{in } \Omega
	\end{array}
\end{equation}
then  for every relatively compact open set $A\subset \Omega$ it holds that
$$
u \ge -G_A f_u + P_A u \text{ a.e}.
$$
\end{lemma}
\proof Let $A$ be a relatively compact open set in $\Omega$. Define $h(x):=u(x)+G_A f_u (x)$, $x\in \R^d$. Note that $u$ is bounded on $A$, hence $f_u$ is also bounded on $A$. 
By combining the proofs of \cite[Lemma 2.2 and Lemma 2.4]{BVW}) we conclude that $G_A f_u \in C(A)$ \footnote{Note that the regularity of the set $A$ in \cite{BVW} is needed only to show that $G_A f_u$ continuously vanishes at $\partial A$.}, which implies that $h\in C(A)$. Further, since $G_A f_u=0$ on $\Omega\setminus A$, we have that $u=h$ on $\R^d\setminus A$. This implies that $P_A u=P_A h$.  Note that on $A$ we have that
\begin{equation*}
-Lh = -L u -L G_A f_u =-L u + f_u
\end{equation*}
in the distributional sense.  Hence, $u$ is a supersolution for \eqref{e:11} if and only if $h$ is superharmonic, that is satisfies \eqref{eq:h_assump}. By Remark \ref{r:superharmonic} it follows that $P_A h(x)\le h(x)$ for almost every $x\in A$ (hence for a.e. $x\in \R^d$). We get that if $u$ is a supersolution for \eqref{e:11}, then
\[
P_A u(x)=P_A h(x) \le h(x)=u(x)+G_A f_u (x),\qquad x\in \Omega.
\] 
\qed

The analogous result to Lemma \ref{l:1} holds for distributional subsolutions to \eqref{e:11}, with the converse inequality. This argument, combined with \cite[Lemma 3.1]{GKL}, implies the following characterization.

\begin{corollary}
    A function $u\in  L^1(\R^d, (1\wedge j(|x|))dx)\cap C(\Omega)$ is a distributional solution of \eqref{e:11}  if and only if for every relatively compact open set $A\subset \Omega$ it holds that 
\begin{equation}\label{e:3}
u = -G_A f_u+P_A u.
\end{equation}
\end{corollary}
Finally, as an immediate consequence of Lemma \ref{l:1} we get a version of the comparison principle for distributional super- and sub-solutions.

\begin{lemma}\label{lem:comparison}
  Let $B$ be a relatively compact open set in $\Omega$ and $u_1, u_2\in  L^1(\R^d, (1\wedge j(|x|))dx)\cap C(\Omega)$ such that $u_1$ is a supersolution and $u_2$ a subsolution to \eqref{e:11}. If $u_1\ge u_2$ on $\R^d\setminus B$ then $u_1\ge u_2$ on $\Omega$. 
 \end{lemma}
 \proof  
Let $A:=\{x\in\Omega:u_2(x)>u_1(x)\}$. Since both $u_1$ and $u_2$ are continuous on $\Omega$, $A$ is open and $A\subset\overline{A}\subset  B\subset \Omega$. Hence, by Lemma \ref{l:1}, for any $x\in A$
\begin{align}
u_2(x)& \le -G_A f_{u_2}(x)+ P_A(u_2 )(x) =-\int_A G_A(x,y)f(u_2(y))dy +\int_{A^c}P_A(x,y)u_2(y) dy \nonumber\\
&\le -\int_A G_A(x,y)f(u_1(y))dy +\int_{A^c}P_A(x,y)u_1(y) dy\le u_1(x).\label{e:max}
\end{align}
In the inequality we used that $f(u_2(y))\ge f(u_1(y))$ for $y\in A$, and $u_2(y)\le u_1(y)$ for $y\in A^c$. The above proves that $A=\emptyset$, that is, $u_1\ge u_2$ on all of $\Omega$.
\qed

\section{Existence of a large solution }

Now using the existence of a supersolution $\overline{u}$ from Lemma \ref{l:supersolution} and the one-sided iteration argument, we arrive to the existence result for solutions to \eqref{eq:problem}.

\begin{thm}\label{thm:solution}
Let $f$ be an increasing $C^1$ function such that $f(0)=0$ satisfying conditions \eqref{eq:f1}, \eqref{eq:KO_refined}, \eqref{eq:growth}, and
\begin{equation}\label{e:int-criterion}
	\int_{0}^1 V(t)f\left(\frac{V(t)}{t}\right)dt<\infty.
\end{equation}
 Then there exists a nonnegative solution $u\in L^1(\R^n)\cap C(\Omega)$ to \eqref{eq:problem} dominated by $\overline{u}$ on $\Omega$.
\end{thm}
\proof
For $k\in\mathbf N$ let $u_k=-G_\Omega f_{u_k}+kM_\Omega1$ be the nonnegative unique weak dual solution to the semilinear problem 
\begin{equation}\label{e:uk}
\begin{array}{rcll}
		-Lu&=& -f(u)& \quad \text{in } \Omega\\
		u&=&0& \quad \text{in }\Omega^c\\
		W_\Omega u&=&k&\quad \text{on }\partial\Omega.
	\end{array}
\end{equation}
The existence of this weak dual solution follows from \cite[Theorem 5.1]{BVW} and assumptions \eqref{e:int-criterion} and \eqref{eq:f1}, latter implying that the function $\Lambda=f$ satisfies the doubling condition \cite[(5.1)]{BVW}. 

First, we show that the sequence $(u_k)_k$ has a nondecreasing subsequence $(u_{k_n})_n$. By \eqref{eq:H1} and \eqref{eq:f1} the function $h(t):= f(\tfrac{V(t)}t)$ satisfies \cite[Condition \textbf{(U)}]{BVW}. Recall that by \eqref{e:martin_estimate} there exists some $c_1>1$ such that for all $x\in\Omega$
\begin{equation*}
c_1^{-1} \le \frac{\delta_\Omega(x)}{V(\delta_\Omega(x))} M_\Omega 1 (x) \le c_1.
\end{equation*}
Since $f$ is increasing, nonnegative and satisfies \eqref{eq:f1}, there exists $c_2>0$ such that $f(c_1 t)\le c_2 f(t)$, $t\ge 0$, and therefore
$$
G_\Omega f(u_k)\leq  G_\Omega f(k M_\Omega1)\le k^{1+M}G_{\Omega}f(M_{\Omega}1)
\leq c_2 k^{1+M}G_\Omega(h\circ\delta_\Omega). 
$$
Therefore, by \cite[Proposition 4.1]{BVW} and \eqref{e:int-criterion} we have that for some $c_3>0$
\begin{align}\label{e:4}
0&\leq\limsup_{x\to\partial\Omega} \frac{\delta_\Omega(x)}{V(\delta_\Omega(x))}G_\Omega f_{u_k}(x)\leq c_2  k^{1+M} \limsup_{x\to\partial\Omega} \frac{\delta_\Omega(x)}{V(\delta_\Omega(x))}G_\Omega(h\circ\delta_\Omega)(x)\nonumber \\
&\leq c_3 k^{1+M} \left(\lim_{x\to\partial\Omega}\int_0^{\delta_\Omega(x)}h(t)V(t)\,\dt+\lim_{x\to\partial\Omega}\delta_\Omega(x)\int_{\delta_\Omega(x)}^{\text{diam}(\Omega)} \frac{h(t)V(t)}t\,\dt\right)=0.
\end{align}
Note that the last equality in \eqref{e:4} follows from \eqref{e:int-criterion} since for $g(t):=f\left(\frac{V(t)}{t}\right) V(t)$ we can apply the dominated convergence theorem to get that
$$
\lim_{\eta\downarrow 0} \eta\int_{\eta}^1 \frac{ g(t)}{t}\, dt= \int_{0}^1 \lim_{\eta\downarrow 0}\frac{ \eta}{t}g(t)1_{[\eta,1]}(t)\, dt=0.
$$
Now by \eqref{e:martin_estimate} and \eqref{e:4} it follows that for $l\ge 2c_1^2k$
\begin{equation}\label{eq:subseq}
\begin{split}
&0<\limsup_{\delta_\Omega(x)\to 0} \frac{\delta_\Omega(x)}{V(\delta_\Omega(x))} u_k(x)=\limsup_{\delta_\Omega(x)\to 0} \frac{\delta_\Omega(x)}{V(\delta_\Omega(x))} (k M_\Omega 1 (x) )\le c_1k \\
&\le c_1^{-1}2^{-1}l \le 2^{-1} \liminf_{\delta_\Omega(x)\to 0} \frac{\delta_\Omega(x)}{V(\delta_\Omega(x))} (l M_\Omega 1(x))\leq 2^{-1}\liminf_{\delta_\Omega(x)\to 0} \frac{\delta_\Omega(x)}{V(\delta_\Omega(x))}u_l(x)<\infty,
\end{split}
\end{equation}
which implies that $u_l>u_k$ on $\Omega_\eta$ for some $\eta$ small enough. Then by Lemma \ref{lem:comparison} it follows that $u_l\ge u_k$ on all of $\Omega$. Now the desired subsequence $(u_{k_n})_n$ is obtained by setting $k_n=(\lfloor 2c_1^2\rfloor +1)^n$.

 Let $\overline{u}\in C(\Omega)$ be the pointwise supersolution of \eqref{e:11} from Lemma \ref{l:supersolution}. Since $G_\Omega 1\in C_0(\Omega)$, by Lemma \ref{lem:integrability} and Lemma \ref{lem:boundary_rate}, $\overline{u}\in L^1(\Omega)$ and 
\begin{equation}\label{e:boundary_supersolution}
    \lim_{x\to \partial\Omega}\frac{\delta_\Omega(x)}{V(\delta_\Omega(x))}\overline{u}(x)=\infty .
\end{equation}
This together with \eqref{eq:subseq} implies that $\overline{u}> u_{k_n}$ near the boundary of $\Omega$, so by Lemma \ref{lem:comparison} we conclude that 
\begin{equation}\label{e:u_k}
\overline{u}\ge u_{k_n},\quad  \text{in }\Omega.
\end{equation}
For $x\in \Omega$ define
$$
u(x):=\lim_{n\to \infty}u_{k_n}(x) \le \overline{u}(x),
$$
and set $u=0$ on $\Omega^c$. Note that $u\in L^1(\Omega)$. Moreover, for all $k_n$,
$$
\liminf_{x\to \partial\Omega}\frac{\delta_\Omega(x)}{V(\delta_\Omega(x))}u(x)\ge \liminf_{x\to \partial\Omega}\frac{\delta_\Omega(x)}{V(\delta_\Omega(x))}u_{k_n}(x)\ge C^{-1} k_n.
$$

Thus $u$ has the required boundary behaviour in \eqref{eq:problem}. Therefore, it remains to check that $u$ is a solution to \eqref{e:11}. Here we use Lemma \ref{l:1}. Let $A$ be any relatively compact open subset of $\Omega$. Then for all $k_n$,
$$
u_{k_n}(x)=-G_A f_{u_{k_n}}(x)+P_A u_{k_n}(x)= -\int_A G_A(x,y)f(u_{k_n}(y))\, dy +\int_{\Omega\setminus A}P_A(x,y)u_{k_n}(y)\, dy.
$$
Since $u_{k_n}$ increases to $u$ and $f$ is nondecreasing, by the monotone convergence theorem we get that
$$
u(x)=-G_A f_u(x)+P_A u(x). 
$$
By \eqref{e:3} it follows that $u$ is a solution to \eqref{e:11}. 
In order to show continuity of $u$, fix any compact $K\subset \Omega$ and choose $\eta>0$ such that $\mathrm{dist}(K, \Omega^c)> 3\eta$. Let $A$ be any $C^{2}$ open subset of $\Omega$ such that $\Omega\setminus \Omega_{2\eta}\subset A \subset \Omega\setminus \Omega_{\eta}$. For any $j\ge 1$ and $x\in A$,
\begin{eqnarray*}
0&\le & u_{k_{n+j}}(x)-u_{n_k}(x)=\big(-G_A f_{u_{k_{n+j}}}(x)+P_A u_{k_{n+j}}(x)\big)-\big(-G_A f_{u_{k_{n}}}(x)+P_A u_{k_{n}}(x)\big)\\
&=&-G_A ( f_{u_{k_{n+j}}}- f_{u_{k_{n}}})(x)+P_A ( u_{k_{n+j}}- u_{k_{n}})(x)\le P_A ( u_{k_{n+j}}- u_{k_{n}})(x).
\end{eqnarray*}
If $x\in K$ and $z\in \Omega\setminus A$, then $|x-z|\ge \eta$. Hence, by \cite[(4.5)]{BVW}, for some $c_4>0$
$$
P_A(x,z)\le \frac{c_4}{V(\delta_{A^c}(z))}, \quad x\in K, z\in \Omega\setminus A,
$$
where $\delta_{A^c}(z)=\text{dist}(z,A)$. This implies that
\begin{equation}\label{e:15}
0\le  u_{k_{n+j}}(x)-u_{n_k}(x)\le c_4\int_{\Omega\setminus  A}\frac{u_{k_{n+j}}(z)- u_{k_{n}}(z)}{V(\delta_{A^c}(z))}\, dz 
\le c_4\int_{\Omega\setminus  A}\frac{u(z)- u_{k_{n}}(z)}{V(\delta_{A^c}(z))}\, dz.
\end{equation}
Since $u\in L^1(\Omega)$ is also locally bounded on $\Omega$, and $\int_0^1 \frac{dr}{V(r)}dr<\infty$, it holds that 
\[
\int_{\Omega\setminus A}\frac{u(z)}{V(\delta_{A^c}(z))}\, dz <\infty.
\]
Therefore, by the dominated convergence theorem, the right-hand side in \eqref{e:15} goes to zero as $n\to \infty$, independently of $x\in K$. Hence, the convergence $u_{k_n}\xrightarrow{n\to\infty} u$ is uniform in $K$, implying that $u$ is continuous on $K$. 
\qed

\begin{remark}
    Note that the solution from Theorem \ref{thm:solution} satisfies
    \[
    \limsup\limits_{x\to\partial\Omega}\frac{u(x)}{\psi(V(\delta(x))}<+\infty,
    \]
    but we do not obtain the exact blow-up rate of the solution at $\partial\Omega$.  In the case of the fractional Laplacian $L=\Delta^{\frac \alpha 2}$, $f(t)=t^p$ and $\Omega=B$ a ball, the blow-up rate of the large solution is known to be of the same order as the supersolution $\overline{u}$, see \cite[Theorem 5.1]{benchrouda}. That is, the solution $u$ is comparable to $\delta_B^{-\frac{\alpha}{p-1}}$ near the boundary $\partial B$. 
\end{remark}

\section{Kato's inequality}

In the classical setting of the Laplace operator, the classical Keller-Osserman condition implies the existence of a universal upper bound on $\Omega$ for distributional solutions to the local semilinear equation, 
\begin{equation}\label{e:11_local}
\Delta u(x)=f(x,u(x)),\quad x\in\Omega,
\end{equation}
see for example \cite[Section 3.1, Theorem 4.1.2]{MV}. This universal upper bound is obtained by applying Kato's inequality and showing that all solutions are dominated by a function given in terms of the constructed supersolution $\overline{u}$. As a consequence, it follows that for every compact set $K\subset \Omega$ there exists a positive constant $C_K$ such that every locally bounded solution of \eqref{e:11_local} satisfies
\begin{equation}\label{e:uniform_bdd}
\sup_{x\in K} |u(x)|\le C_K.
\end{equation}
This uniform local boundedness is then the key ingredient in the construction of the maximal solution to \eqref{e:11_local}. Although it is still unclear whether similar consequences of the generalised Keller-Osserman condition hold in the nonlocal setting, even in the case of the fractional Laplacian, we investigate the corresponding distributional Kato's inequality and one of its direct implications in the context of semilinear equations.

\begin{prop}\label{prop:kato-distrbution}
Let $\Omega$ be a bounded $C^{1,1}$ open set, $F\in L^1_\text{loc}(\Omega)$ and $u\in L^1(\rd, (1\wedge j(|x|))dx)$ the distributional solution to the linear problem
\begin{equation}\label{e:Kato2}
\begin{array}{rcll}
		-Lu&=& F& \quad \text{in } \Omega.
	\end{array}
\end{equation}
Then for every $\xi\in C_c^\infty(\Omega)$, $\xi\ge 0$,
\begin{align}
&\int_{\R^d}|u|(-L\xi)\leq \int_{\Omega}\xi \text{sgn}(u)F.\label{e:kato3} 
\end{align}
\end{prop}
We first prove a simple integration by parts type formula.
\begin{lemma} \label{lemma:ibp}
Let $\epsilon_0>0$ and $\xi \in C^{\infty}_c(\Omega\setminus \Omega_{\epsilon_0})$. For any $0<\epsilon<\frac{\epsilon_0}{2}$ and $v\in C^2(\Omega\setminus\Omega_{\epsilon}) \cap L^1(\rd, (1\wedge j(|x|)dx))$, it holds that 
\begin{align} \label{dist-kato-IBP}
    \int_{\rd} v L\xi = \int_{\Omega\setminus\Omega_{\epsilon_0}} \xi Lv . 
\end{align}
\end{lemma}

\begin{proof}
For $\delta>0$ define the approximation operator 
\begin{align*}
    L_{\delta}\xi (x) = \int_{|y|>\delta} \big(\xi(x+y) -\xi(x)\big) j(|y|) dy, 
\end{align*}
By the symmetry of the operator $L_{\delta}$ we have 
\begin{align}\label{dist-kato-IBP-approx}
    \int_{\rd} v L_{\delta}\xi = \frac 1 2\int_{\rd}\int_{\rd}(v(z)-v(x))(\xi(x)-\xi(z))j(|z-x|)1_{\{|z-x|>\delta\}}dz dx=\int_{\Omega} \xi L_{\delta} v . 
\end{align}
Note that $\lim_{\delta\rightarrow 0} \int_{\rd} v L_{\delta}\xi = \int_{\rd} v L\xi$ by the dominated convergence theorem. Indeed, by assumption $v\in L^1(\rd, (1\wedge j(|x|))dx)$, and since $\xi \in C^{\infty}_c(\Omega\setminus\Omega_{\epsilon_0})$ we have that $|L_{\delta} \xi(x)|\leq C((1\wedge j(|x|))dx)$ for some constant $C$ independent of $\delta$. 

Furthermore, as $v\in C^2(\Omega\setminus\Omega_{\epsilon})$ we have $\|D^2 v\|_{L^{\infty}(\Omega\setminus\Omega_{\epsilon_0/2})}<+\infty$. For any $x\in \Omega\setminus\Omega_{\epsilon_0}$ we get
\begin{align*}
     |L_{\delta} v(x)|&\leq \int_{\delta< |y|<\epsilon} \big|v(x+y) -v(x)-\nabla v(x)\cdot y\big| j(|y|) dy \, + \big|L_{\epsilon} v(x) \big|\\ 
    & \leq \|D^2 v\|_{L^{\infty}(\Omega\setminus\Omega_{\epsilon_0/2})} \int_{|y|<\epsilon} |y|^2 j(|y|) dy + |L_{\epsilon} v(x)| := h_\epsilon(x).  
\end{align*}
Since $h_\epsilon \in L^1_{loc}(\Omega_\epsilon)$, by the dominated convergence theorem we have $\lim_{\delta\rightarrow 0} \int_{\Omega} \xi L_{\delta} v = \int_{\Omega} \xi L v$, so the result follows by taking limit as $\delta \rightarrow 0$ in \eqref{dist-kato-IBP-approx}.
\end{proof}

\begin{proof}[Proof of Proposition \ref{prop:kato-distrbution}]
 We prove the result in several steps.  

\textit{Step1:} 
Let $\{\rho_\epsilon\}_{\epsilon>0}$ be a family of mollifiers such that $\int_{\rd} \rho_\epsilon =1$ and define $u_{\epsilon} = u * \rho_\epsilon$ and $F_{\epsilon}:= F * \rho_{\epsilon}$. Fix $\epsilon_0>0$. First we show that $u_{\epsilon}$ is a distributional (and pointwise) solution of 
\begin{align}\label{proof:kato-dist-eqn-reg}
    -L u_{\epsilon} = F_{\epsilon} \quad \text{in} \quad \Omega\setminus\Omega_{\epsilon_0}. 
\end{align}
for all $\epsilon< \frac{\epsilon_0}{2}$. Note that for every $\xi\in C_c^\infty(\Omega\setminus\Omega_{\epsilon_0})$, $\epsilon< \frac{\epsilon_0}{2}$ and $y\in B_\epsilon(0)$ we have that $\overline{\xi}_y:= \xi (\cdot + y) \in C^{\infty}_c(\Omega)$. Since $u$ is a distributional solution to \eqref{e:Kato2} we get that
\begin{align*}
    -\int_{\rd} u_{\epsilon} \, L \xi& = -\int_{\rd} \int_{\rd} \rho_{\epsilon}(y) u(x-y) L\xi(x) \, dx \, dy = -\int_{\rd} \rho_{\epsilon}(y)  \int_{\rd} u(z) L\overline{\xi}_y(z) dz \, dy =\nonumber\\
    &=\int_{\rd} \rho_\epsilon (y) \int_{\Omega} F(z) \overline{\xi}_y(z) dz dy
    =\int_{\Omega\setminus\Omega_{\epsilon_0}} \xi(x) \int_{\rd} F(x-y) \rho_{\epsilon}(y) dy \, dx\nonumber\\
    &= \int_{\Omega\setminus\Omega_{\epsilon_0}} F_{\epsilon} \, \xi.
\end{align*}
\textit{Step2:} Let $0\leq \xi \in C^{\infty}_c(\Omega)$. Then there exists $\epsilon_0>0$ such that $\text{supp}(\xi) \subset \Omega\setminus\Omega_{\epsilon_0}$. 
By applying Taylor's theorem to a convex function $\gamma \in C^2 (\R)$, for $x\in \rd$ and $\epsilon\le \epsilon_0$ we get that
\begin{align}\label{proof:kato-dist-libn}
{ - L(\gamma(u_\epsilon))(x)} &  { =   - \gamma'(u_\epsilon)(x) L u_\epsilon(x) -  \text{P.V. }\int_{\R^d} \frac{\gamma''(z_{x,y,u_\epsilon})}{2}(u_\epsilon(x)-u_\epsilon(y))^2 j(|x-y|) dy }\notag \\  
   & \leq - \gamma'(u_\epsilon)(x) L u_\epsilon(x).
\end{align}
By Lemma \ref{lemma:ibp},  \eqref{proof:kato-dist-libn} and \eqref{proof:kato-dist-eqn-reg}, we get that  
\begin{align}\label{proof:kato-dist-eps-ineq}
    \int_{\rd}  \gamma(u_\epsilon) \big(-L\xi\big)= \int_{\Omega}  -L(\gamma(u_\epsilon))\xi \leq \int_{\Omega}  \gamma'(u_\epsilon) \big(-L u_\epsilon\big) \xi = \int_{\Omega} \gamma'(u_\epsilon)F_\epsilon \, \xi. 
\end{align}

\textit{Step 3:} As $u \in L^1(\rd, 1\wedge j(|\cdot|))$ one can easily show that $u_\epsilon \xrightarrow{\epsilon\downarrow 0} u$ in $L^1(\rd, 1\wedge j(|\cdot|))$, see \cite[Lemma 2.9]{GKL}. Therefore, for $\gamma \in C^2(\R)$ such that $\gamma'$ is bounded it follows that 
\begin{align*}
    \gamma(u_\epsilon) \xrightarrow{\epsilon\downarrow 0} \gamma(u) \quad \text{in} \quad  L^1(\rd, 1\wedge j(|\cdot|)).  
\end{align*}
Since for $\xi \in C^{\infty}_c(\Omega)$ we have that $|L\xi(x)| \lesssim 1\wedge j(|x|) $, the previous convergence implies 
\begin{align}\label{eq:lhs}
    \lim_{\epsilon\rightarrow 0} \int_{\rd} \gamma(u_\epsilon) (-L\xi) = \int_{\rd} \gamma(u) (-L\xi).
\end{align}
On the other hand, note that similarly as in \eqref{proof:kato-dist-eqn-reg},
\begin{align}\label{eq:rhs}
    \int_{\Omega} \gamma'(u_\epsilon) F_\epsilon \, \xi = \int_{\Omega} \gamma'(u_\epsilon)(x)  \, \xi(x) \int_{\rd} F(y) \rho_\epsilon(x-y) dy \, dx = \int_{\Omega} F\,  \big( \rho_\epsilon * (\gamma'(u_\epsilon)\xi)\big).
\end{align}
Since $u_\epsilon \rightarrow u$ in $L^1(\rd, 1\wedge j(|\cdot|))$ it follows that (on a subsequence) $u_\epsilon \longrightarrow u$ a.e. Since $\gamma'\in C_b(\R)$, by dominated convergence theorem we get that 
\begin{align}\label{proof:kato-dist5}
    \int_{\Omega} \big| (\gamma'(u_\epsilon) - \gamma'(u))\big|  \xi   \xrightarrow{\epsilon\downarrow0} 0. 
\end{align}
We apply the following decomposition 
\begin{align}\label{proof:kato-dist6}
    \Big|\int_{\Omega}F\,  \big( \rho_\epsilon * (\gamma'(u_\epsilon)\xi)\big) - \int_{\Omega} F \gamma'(u)\xi \Big| & \leq \Big| \int_{\Omega}F\,  \big( \rho_\epsilon * (\gamma'(u_\epsilon)\xi)\big) - \int_{\Omega} F \big( \rho_\epsilon * (\gamma'(u)\xi)\big)\Big|  \notag \\
    & \qquad + \Big|  \int_{\Omega} F \big( \rho_\epsilon * (\gamma'(u)\xi)\big) - \int_{\Omega} F \gamma'(u)\xi \Big|  \notag \\
    & := A_\epsilon + B_\epsilon. 
\end{align}
Since $\xi \in C^{\infty}_c (\Omega\setminus\Omega_{\epsilon_0})$ and $\rho_\epsilon \in C^{\infty}_c(B_\epsilon(0))$, we get that $\text{supp}\Big(\rho_\epsilon * (\gamma'(u_\epsilon)\xi) - \rho_\epsilon * (\gamma'(u)\xi) \Big) \subset \text{supp}(\rho_\epsilon) + \Omega\setminus\Omega_{\epsilon_0} \subset K_{\xi}$ for some compact set $K_{\xi}\subset \Omega$ independent of $\epsilon>0$. Then by $\|\rho_\epsilon\|_{L^{\infty}(\rd)}\leq 1$ and \eqref{proof:kato-dist5} we have that
\begin{align*}
    A_\epsilon & = \Big| \int_{K_{\xi}} F(x)\,   \rho_\epsilon * \big((\gamma'(u_\epsilon)-\gamma'(u))\xi\big)(x) \, dx \Big|\\
    & = \Big|    \int_{K_{\xi}}  F(x)  \Big( \int_{B_\epsilon(x)} \rho_\epsilon(x-y) (\gamma'(u_\epsilon) - \gamma'(u))(y) \xi(y) dy \Big) dx \Big| \\
    & \leq \|F\|_{L^1(K_{\xi})} \int_{\Omega} \big| (\gamma'(u_\epsilon) - \gamma'(u))  \xi \big| \,  \xrightarrow{\epsilon\rightarrow 0} 0. 
\end{align*}
Since the family of mollifiers is uniformly bounded, by boundedness of $\xi$ and $\gamma'$, for $x\in K_\xi$  it follows that
$$|F(x) \rho_\epsilon* (\gamma'(u)\xi)(x)|\leq \|\gamma'(u)\xi\|_{L^{\infty}(\Omega)}|F(x)| . $$  
Since $F\in L^1(K_\xi)$ and $\rho_\epsilon* (\gamma'(u)\xi) \rightarrow (\gamma'(u)\xi)$, by the dominated convergence theorem we get that $\lim\limits_{\epsilon \rightarrow 0} B_\epsilon = 0$. Therefore, by taking the limit in the inequality \eqref{proof:kato-dist-eps-ineq}, and applying \eqref{eq:lhs}, \eqref{eq:rhs} and \eqref{proof:kato-dist6} we get that  
\begin{align}\label{proof:kato-dist-reg-ineq}
    \int_{\rd}  \gamma(u) \big(-L\xi\big) \leq  \int_{\Omega} \gamma'(u)F \, \xi. 
\end{align}

\textit{Step 4:} Let $(\gamma_{n})_{n}$ be an increasing sequence of nonnegative $C^2$ convex functions such that $ \gamma_n(x) \rightarrow |x|$, $\gamma_n'(x) \rightarrow \text{sgn}(x)$ as $n \rightarrow \infty$ and $\|\gamma_n'\|_{L^{\infty}(\R)}<K$ for some constant $K>0$ independent of $n$. By \eqref{proof:kato-dist-reg-ineq},
\begin{align}\label{proof:kato-dist-reg-ineq-approx}
    \int_{\rd}  \gamma_n(u) \big(-L\xi\big) \leq  \int_{\Omega} \gamma_n'(u)F \, \xi. 
\end{align}
Since the sequence $(\gamma_n)_n$ is increasing, by splitting $L\xi$ into the positive and negative part, we can apply the monotone convergence theorem to find that 
\begin{align*}
    & \lim_{n\to\infty} \int_{\rd}  \gamma_n(u) \big(-L\xi\big) = \lim_{n\to\infty} \Big( \int_{\{L\xi>0\}} + \int_{\{L\xi\leq 0\}}\Big) \gamma_n(u)(x) \big(-L\xi\big)(x) \, dx \\
    & = \Big( \int_{\{L\xi>0\}} + \int_{\{L\xi\leq 0\}}\Big)  |u(x)| \big(-L\xi\big)(x) \, dx = \int_{\rd} |u(x)| \big(-L\xi\big)(x) \, dx . 
\end{align*}
On the other hand, as $|\gamma_n'|\leq K$ and $F\in L^1_{loc}(\Omega)$, the dominated convergence theorem implies that 
\begin{align*}
    \lim_{n\to\infty} \int_{\Omega} \gamma_n'(u)F \, \xi = \int_{\Omega} \text{sgn}(u)F \, \xi.
\end{align*}
The proof is finished by taking the limit as $n\to\infty$  in \eqref{proof:kato-dist-reg-ineq-approx}.

\end{proof}

A straightforward consequence of Kato's inequality is the fact that given two distributional solutions $u$ and $v$ to the semilinear problem \eqref{e:11}, their maximum
\[
w:=\max\{u,v\}=\frac{u+v+|u-v|}{2}
\]
is a subsolution to \eqref{e:11}.
Since $f$ is increasing, by linearity of the operator $L$ and \eqref{e:Kato2} we get the following inequality in the distributional sense
\begin{align*}
-Lw&=-\frac{f(u)+f(v)}{2}-\frac 1 2 L(|u-v|)\\
&\le -\frac{f(u)+f(v)}{2}-\frac{\text{sgn}(u-v)(f(u)-f(v))}2\\
&= -\frac{f(u)+f(v)}{2}-\frac{|f(u)-f(v)|}2\\
&= -\max(f(u),f(v))=-f(w),  
\end{align*}
that is, $w$ is a subsolution to \eqref{e:11}. Analogously, $\min\{u,v\}$ is a supersolution to \eqref{e:11}.

\vspace{1cm}

\noindent {\bf Acknowledgments.} The authors would like to thank Dr.~Ivan Biočić for thoughtful reading and valuable comments on the manuscript.

	\bigskip
	
	\noindent{\bf Indranil Chowdhury}
	
	\noindent Indian Institute of Technology - Kanpur (IIT Kanpur), Kanpur, India,
	
	\noindent Email: \texttt{indranil@iitk.ac.in}

    \bigskip
    
  \noindent{\bf Zoran Vondra\v{c}ek} 

  \noindent Dr. Franjo Tu\dj man Defense and Security University, Zagreb, Croatia, and

  \noindent Faculty of Science, Department of Mathematics, University of Zagreb, Zagreb, Croatia,

  \noindent Email: \texttt{vondra@math.hr}

	\bigskip
	
	\noindent{\bf Vanja Wagner}
	
	\noindent Faculty of Science, Department of Mathematics, University of Zagreb, Zagreb, Croatia,
	
	\noindent Email: \texttt{wagner@math.hr}%
\end{document}